\documentclass{amsart}
\usepackage[margin=0.8in]{geometry}

\usepackage{amsmath, amsfonts, amssymb, mathtools}

\usepackage{microtype}
\usepackage[hidelinks]{hyperref}
\usepackage{xcolor}

\usepackage{graphicx}
\usepackage{booktabs}
\usepackage{float}
\usepackage{subcaption}
\usepackage[export]{adjustbox}

\usepackage{verbatim}

\usepackage[maxbibnames=99]{biblatex}
\mathtoolsset{showonlyrefs}
\numberwithin{equation}{section}

\newtheorem{theorem}{Theorem}[section]
\newtheorem{lemma}[theorem]{Lemma}

\theoremstyle{definition}

\theoremstyle{remark}

\newcommand{\R}{\mathbb{R}}

\newcommand{\score}{\nabla \log}

\newcommand{\KL}[2]{\ensuremath{\operatorname{KL}\!\left(#1\,\|\,#2\right)}}

\newcommand{\intr}{\int_{\R^d}}
\newcommand{\intrthree}{\int_{\R^3}}

\newcommand{\f}{f}

\renewcommand{\v}{v}

\newcommand{\lpnorm}[2]{\left\| #1 \right\|_{L^{#2}}}

\newcommand{\linftynorm}[1]{\lpnorm{#1}{\infty}}
\newcommand{\vnorm}[1]{\left| #1 \right|}
\newcommand{\jb}[1]{\langle #1 \rangle}
\newcommand{\lp}{\left(}
\newcommand{\rp}{\right)}

\renewcommand{\d}{\,\mathrm{d}}

\title[Stability of Landau in relative entropy]{Stability of the spatially homogeneous Landau equation in relative entropy and applications to score-based numerical methods}

\author{Vasily Ilin}
\address{Department of Mathematics, University of Washington, Seattle, WA, USA}
\email{vilin@uw.edu}


\begin{document}

\begin{abstract}
We give a short and elementary proof of stability for strong solutions of the spatially homogeneous Landau equation with Coulomb collisions, measured in relative entropy. The argument yields an explicit differential inequality for relative entropy under natural moment and regularity assumptions. The same computation provides an a posteriori error bound for score-based transport modeling and related deterministic numerical schemes, linking the training loss to the relative-entropy error.
\end{abstract}

\maketitle

\section{Introduction}
The spatially homogeneous Landau equation describes the time evolution of a velocity distribution $f(t,v)$ of a spatially homogeneous plasma:
\begin{equation}\label{eq:Landau}
    \partial_t f = \nabla_v \!\cdot \!\int_{\R^d} A(v-w)\big[f(w)\nabla_v f(v) - f(v)\nabla_w f(w)\big]\,\mathrm{d}w,
\end{equation}
with
\[
    A(z) = |z|^\gamma (|z|^2 I_d - z \!\otimes\! z), \qquad -d-1 \le \gamma \le 1.
\]
We focus on the physically relevant Coulomb case $d=3$, $\gamma=-3$. Beyond classical existence and a priori bounds, a key qualitative question is \emph{stability}: how does a solution respond to perturbations of the initial data?  We provide a short, entropy-based proof of stability for strong solutions in relative entropy $\KL{f_t}{g_t}$, relying on a single elementary computation that exploits the continuity form of the equation and the coercivity of $A\ast g$.  This yields a differential inequality for $\KL{f_t}{g_t}$ with coefficients controlled by moments, $L^p$ norms, and a weighted Fisher information, together with growth bounds for the latter along the flow.  The same computation also provides an \emph{a posteriori} error bound for score-based numerical solvers, such as SBTM \cite{ilin2024transport,huang2025score} and the blob method \cite{carrillo2020landau}, linking training loss to KL error through a computable inequality.

\paragraph{Related Work and Contributions.}
The homogeneous Landau equation has been extensively studied since Villani’s foundational works \cite{villani1998spatially,villani2000decrease,Villani02}, which developed the entropy–Fisher-information framework and proved convergence to equilibrium for Maxwellian and moderately soft potentials.  For the Coulomb case, Guillen and Silvestre \cite{guillen2023landau} established global existence and Gaussian bounds, later extended to smooth solutions in $L^{3/2}$ by Golding, Gualdani, and Loher \cite{golding2024global}.  Coercivity estimates for $A\ast f$ were obtained in \cite{alexandre2013some}, forming the analytic foundation of modern regularity theory.  Weak–strong uniqueness in relative entropy was recently proved by Tabary \cite{tabary2025weak} via a dimension-doubling argument in $(v,v')\in\R^6$; in contrast, our proof remains entirely in velocity space and proceeds by an energy-type computation on $\KL{f}{g}$, avoiding tensorization and yielding explicit quantitative bounds.  

Our approach combines three standard analytic ingredients—coercivity of $A\ast g$, propagation of moments and $L^p$ regularity for soft potentials \cite{golding2024global}, and functional inequalities for weighted Fisher information \cite{toscani2000trend}—into a single direct estimate.  The resulting framework unifies analytic stability for the Landau flow with the score-matching objective of deterministic numerical solvers \cite{boffi2023probability,ilin2024transport,LWX24,huang2025score,ilin2025score}. Specifically, we prove: (i) a concise entropy computation establishing KL-stability for strong solutions, and (ii) a KL error bound for score-based numerical Landau methods in terms of the weighted score-matching loss.

\paragraph{Organization.}
Section~\ref{sec: prior} recalls the required a priori estimates and coercivity results. Section~\ref{sec: key estimates} derives Fisher-information bounds and presents the key computation. Section~\ref{sec: main} establishes KL-stability and the numerical error bound (Theorems~\ref{thm: Landau stability}–\ref{thm: numerical method error bound}).

\section{Known results}\label{sec: prior}
We first recall estimates from Section 23 in \cite{villani2025fisher} combined with Theorem 0.1 in \cite{golding2024global}. 
\begin{lemma}[Regularity of solutions]\label{lemma: regularity of solutions} 
Let $f_t$ satisfy the Landau equation with $-4 < \gamma < -2$ and $\gamma \ge -d$. Assume the initial data $f_0$ has finite energy, entropy, Fisher information and enough moments:
\begin{align}
    \intr |v|^2 f_0(v)\d v < \infty, \quad \intr f_0(v)\log f_0(v)\d v < \infty, \quad \intr |\score f_0|^2 f_0 \d v < \infty, \quad \intr |v|^s f_0(v)\d v < \infty.
\end{align}
Then the solution to the Landau equation $f_t$ is unique, smooth, and satisfies
\[
\begin{split}
      &   \intr \jb{v}^s f_t(v) \d v \leq C(s,\gamma,f_0)(1+t), \\
      &   \forall \varepsilon>0, \lpnorm{f_t}{p} \leq C(p,d,\gamma,f_0,\varepsilon)(t^{-\alpha_p(p,d,\gamma)} + t^\varepsilon),\\
        \end{split}
        \]
    and for any $k,\kappa>0$, and any $\varepsilon>0$, there exist $\alpha>0$ and $s>0$ such that for all $t>0$,
    \begin{align}
        |\nabla^k f_t(v)| \leq \frac{C(f_0)(t^{-\alpha} + (1+t)^\varepsilon)}{\jb{v}^\kappa}. 
    \end{align}
    If we assume further that $\lpnorm{f_0 \jb{v}^{\frac{9}{2p}}}{p} < \infty$ with $p\geq \frac{3}{2}$, then 
    \begin{align}
        \lpnorm{f}{\infty} \lesssim C(f_0) t^{-\frac{3}{2p}}.
    \end{align}
\end{lemma}

We rephrase proposition 2.1 in \cite{alexandre2013some} as the following lemma. 
\begin{lemma}[Coercivity of $A\ast \f$]\label{lemma: coercivity of A*f}
    Suppose that $\f$ is a  probability density with finite energy and entropy, i.e.
    \begin{align}
       \frac 1 2  \intr\vnorm{\v}^2 \f \d\v = E < \infty, \quad  \intr\f  \log \f  \d\v  = H < \infty. 
    \end{align}
    Then there exists a constant $C_{coe} > 0$ depending only on $\gamma, E$ and $H$ such that
    \begin{align}
        A \ast \f  \geq C_{coe} \jb{\v}^\gamma. 
    \end{align}
\end{lemma}

We state the $L^2$ version of Pinsker's inequality from \cite{wang2017mean}.
\begin{lemma}[$L^2$ version of Pinsker's inequality]\label{lemma: L2 Pinsker}
Assume that $f$ and $g$ are two probability densities that are also in $L^2(\R^d) \cap L^1(\R^d) \cap L^\infty(\R^d)$. Then \begin{align}
        \intr |f-g|^2 \d \v \leq 2 \lp \linftynorm{f} + \linftynorm{g} \rp \KL{f}{g}, \quad \KL{f}{g} := \intr f \log \frac{f}{g} \d \v.
    \end{align}
\end{lemma}
\begin{proof}
    Take $g(x) = x\log x$. Then by Taylor's expansion
    \begin{align}
        g(f) - g(g) = (1+\log g)(f-g) + \frac{1}{2}\frac{1}{\zeta}(f - g)^2, 
    \end{align}
    for some $\zeta$ between $f$ and $g$. Taking integral of both sides,
    \begin{align}
        \intr f \log \frac{f}{g} \d \v = \frac{1}{2} \intr \frac{1}{\zeta}(f - g)^2 \d \v \geq \frac{1}{2} \frac{1}{\|\zeta\|_{L^\infty}}\intr (f - g)^2 \d \v \geq \frac{1}{2} \frac{1}{\linftynorm{f} + \linftynorm{g}}\intr (f - g)^2 \d \v. 
    \end{align}
    Thus,
    \begin{align}
        \intr (f - g)^2 \d \v \leq 2 \lp \linftynorm{f} + \linftynorm{g} \rp \intr f \log \frac{f}{g} \d \v.
    \end{align}
\end{proof}

\section{Key estimates}\label{sec: key estimates}
The spatially homogeneous Landau equation can be rewritten in the form of a continuity equation as follows
\begin{equation*}
    \partial_t f + \nabla\cdot (f U[f]) = 0, 
\end{equation*}
where the velocity field $U[f]$ reads 
\[
    U[f](v) = -\intr A(v-w)(\score f(v) - \score f(w))f(w)\d w 
    \]
and the matrix is given by 
\[
    A(z) = |z|^\gamma (|z|^2I_d - z \otimes z).
   \] 
This is the form of the equation that we will use for the results below. In the rest of the paper $f$ denotes the solution to the Landau equation with $d=3$ and $\gamma=-3$, and $f_0$ satisfies the energy, entropy, Fisher information and moment finiteness assumptions from lemma \ref{lemma: regularity of solutions}.

\begin{theorem}[Growth estimates on weighted Fisher information]\label{thm: growth of weighted fisher}
    For any $s > 2$, it holds that 
    \begin{align}
        \intrthree f_t(v) \jb{v}^s |\score f_t|^2 \d v \leq C (t^{-\alpha} + 1+t)
    \end{align}
    for some $C = C(\gamma,s,f_0)>0$ and $\alpha = \alpha(\gamma, s, f_0) > 0$. 
    If we assume further that $\| f_0 \|_{H^2_5}^2 < \infty$, then it holds that 
    \begin{align}
        \intrthree f_t(v) \jb{v}^s |\score f_t|^2 \d v \leq C (1+t).
    \end{align}
\end{theorem}
\begin{proof}
  By the estimates shown in page 1299 in \cite{toscani2000trend} as part of proposition 5, one has the following functional inequality 
\begin{align}
    \intrthree f(v) \jb{v}^s |\score f|^2 \d v \lesssim \| f \|_{H^2_{s+2}} + \intrthree \jb{v}^{s-2} f(v)\d v,
\end{align}
where 
\begin{align}
    \| f \|_{H^2_{s+2}}^2 := \lpnorm{|\nabla^2 f(v)| \jb{v}^{s+2}}{2}^2 + \lpnorm{|\nabla f(v)| \jb{v}^{s+2}}{2}^2 + \lpnorm{|f(v)| \jb{v}^{s+2}}{2}^2.
\end{align}
Using the pointwise estimate
\begin{align}
    |\nabla^k f_t| \lesssim \frac{t^{-\alpha} + (1+t)}{\jb{v}^{\kappa}}
\end{align}
and the moment estimate
\begin{align}
    \intrthree \jb{v}^{s-2} f_t(v)\d v \lesssim 1+t
\end{align}
from lemma \ref{lemma: regularity of solutions} with large enough 
$\kappa$ gives
\begin{align}
    \| f_t \|_{H^2_5}^2 \lesssim t^{-2\alpha} + (1+t)^2
\end{align}
and 
\begin{align}
    \intrthree f_t(v) \jb{v}^s |\score f_t|^2 \d v \lesssim t^{-\alpha} + 1+t.
\end{align}
Further, if $\| f_0 \|_{H^2_5}^2 < \infty$, then
    \begin{align}
    \| f_t \|_{H^2_5}^2 \lesssim (1+t)^2
\end{align}
and 
\begin{align}
    \intrthree f_t(v) \jb{v}^s |\score f_t|^2 \d v \lesssim 1+t.
\end{align}
\end{proof}

The following lemma demonstrates a simple condition on the initial data that is sufficient to guarantee the main results.
\begin{lemma}[Sufficient assumption on initial data]\label{lemma: sobolev norm of initial data implies finiteness for all t}
     Assume the initial finiteness of the following weighted Sobolev norm of the solution $f$:
    \begin{align}
      &  \| f_0 \|_{H^2_5}^2 := \intrthree|\nabla^2 f_0(v)|^2 \jb{v}^{10} \d v+ \intrthree |\nabla f_0(v)|^2 \jb{v}^{10}\d v + \intrthree|f_0(v)|^2 \jb{v}^{10} \d v < \infty.
    \end{align}
    Then we have the following growth estimates:
    \begin{align}
        \lpnorm{f \jb{v}^3}{2} \lesssim 1+t, \quad
        \lpnorm{f \jb{v}^3}{1} &\lesssim 1+t, \quad
        \lpnorm{f}{\infty} \lesssim t^{-\frac{3}{4}}, \quad
        \lpnorm{f \jb{v}^3 |\score f|^2}{1} \lesssim 1+t.
    \end{align}
\end{lemma}

\begin{proof}
    By Cauchy-Schwarz and lemma \ref{lemma: regularity of solutions}:
    \begin{align}
        \lpnorm{f \jb{v}^3}{2} &\leq \lpnorm{f}{3}^{3/4} \cdot \lpnorm{f \jb{v}^6}{1}^{1/4} \lesssim (t^{-\alpha} + t^\varepsilon)\cdot(1+t)^{1/4}\\
        \lpnorm{f \jb{v}^3}{1} &\lesssim 1+t.
    \end{align}
    Assuming $\lpnorm{f_0}{3}<\infty$, and choosing $\varepsilon=3/4$, gives
    \begin{align}
        \lpnorm{f \jb{v}^3}{2} \lesssim 1+t.
    \end{align}
    Taking $p=2$ in the last inequality of lemma \ref{lemma: regularity of solutions} and assuming $\lpnorm{f_0}{3}<\infty$, so that $\lpnorm{f_0 \jb{v}^{\frac{9}{4}}}{2} < \infty$, we get 
    \begin{align}
        \lpnorm{f}{\infty} \lesssim t^{-\frac{3}{4}}.
    \end{align}
    Lastly, taking $s=3$ in theorem \ref{thm: growth of weighted fisher} and assuming $\|f_0\|_{H^2_5} < \infty$, we obtain
    \begin{align}
        \lpnorm{f \jb{v}^3 |\score f|^2}{1} \lesssim 1+t.
    \end{align}
    By Sobolev embedding
    \begin{align}
        \lpnorm{f_0}{3} \lesssim \lpnorm{f_0 |\score f_0|^2}{1} \leq \lpnorm{f_0 \jb{v}^3 |\score f_0|^2}{1} \lesssim \|f_0\|_{H^2_5}
    \end{align}
\end{proof}

The following is the key computation of this work. Neither $f$ nor $g$ is assumed to solve the Landau equation.
\begin{lemma}\label{lemma: key calculation}
For any $d$ and $\gamma \leq 0$, and any $f$ and $g$ with finite energy and entropy, for which the LHS and the RHS are defined and finite, the following inequality holds:
    \begin{align}
        \intr f (U[f] - U[g]) \cdot \score \frac{f}{g} \leq& -\frac{1}{2} \intr f\left|\score \frac{f}{g} \right|_{A\ast g}^2 \\
        &+ C_{coe}\intr f \jb{v}^{-\gamma}\|A\ast (f-g)\|^2 |\score f|^2 + C_{coe} \intr f \jb{v}^{-\gamma} |b\ast(f-g)|^2. 
    \end{align}
where
    \begin{align}
        A(z) = |z|^\gamma (|z|^2I_d - z \otimes z), \quad     b(z) = \nabla \cdot A(z) = -2 z|z|^\gamma.
    \end{align}
Further, if $d=3$ and $\gamma=-3$, one has the following inequality:
    \begin{align}
        \intr f (U[f] - U[g]) \cdot \score \frac{f}{g} \leq& -\frac{1}{2} \intr f\left|\score \frac{f}{g} \right|_{A\ast g}^2 \\
        &+ C \KL{f}{g} \Big( \left(\lpnorm{f}{\infty}+\lpnorm{g}{\infty}\right) \left(\lpnorm{f \jb{v}^3 |\score f|^2}{1} + \lpnorm{f \jb{v}^3}{2}\right) + \lpnorm{f \jb{v}^3}{1} \Big)
    \end{align}
for some constant $C$ that depends only on the energy and entropy of $f$.
\end{lemma}
\begin{proof}
    Note that
\begin{align}
    U[f](v)
    &= -\intr A(v-w)\score f(v) f(w)\d w + \intr A(v-w)\nabla f(w)\d w\\
    &= -(A\ast f)\score f(v) + (b\ast f)(v), 
\end{align}
where
\begin{align}
    A(z) = |z|^\gamma (|z|^2I_d - z \otimes z), \quad     b(z) = \nabla \cdot A(z) = -2 z|z|^\gamma.
\end{align}
Hence,
\begin{align}
    \frac{\d}{\d t}\KL{f}{g} = -\intr f\score \frac{f}{g} \cdot \left( (A\ast f)\score f - (A\ast g)\score g \right)  + \intr f \score \frac{f}{g} \cdot b\ast (f-g).
\end{align}
By adding and subtracting the cross-term, one has 
\begin{align}
    (A\ast f)\score f - (A\ast g)\score g = (A \ast g) \score \frac{f}{g} + (A\ast(f-g))\score f.
\end{align}
So far we have the exact expression
\begin{align}
    \frac{\d}{\d t}\KL{f}{g} = -\intr f\left|\score \frac{f}{g} \right|_{A\ast g}^2  - \intr f \score \frac{f}{g} \cdot (A\ast(f-g))\score f  + \intr f \score \frac{f}{g} \cdot b\ast (f-g).
\end{align}
Now we start estimating. By multiplying and dividing by $A\ast g$, and using Cauchy-Schwarz and Young's inequalities, one has
\begin{align}
    \left|\intr f \score \frac{f}{g} \cdot (A\ast(f-g))\score f \right| 
    &\leq \frac{1}{4}\intr f\left| 
\score \frac{f}{g} \right|^2_{A\ast g} + \intr f |A\ast (f-g) \score f|^2_{(A\ast g)^{-1}}, \\
    \intr f \score \frac{f}{g} \cdot b\ast (f-g) 
    &\leq \frac{1}{4} \intr f \left| \score \frac{f}{g} \right|^2_{A\ast g} + \intr f |b\ast(f-g)|^2_{(A\ast g)^{-1}}.
\end{align}
Hence,
\begin{align}
    \frac{\d}{\d t}\KL{f}{g} \leq -\frac{1}{2} \intr f\left|\score \frac{f}{g} \right|_{A\ast g}^2 + \intr f |A\ast (f-g) \score f|^2_{(A\ast g)^{-1}} + \intr f |b\ast(f-g)|^2_{(A\ast g)^{-1}}.
\end{align}
Given that $g$ has finite energy and entropy, by lemma \ref{lemma: coercivity of A*f}, one has 
\begin{align}
    A \ast g \geq C_{coe}\jb{v}^\gamma.
\end{align}
Hence, we obtain
\begin{align}
    \frac{\d}{\d t}\KL{f}{g} \leq& -\frac{1}{2} \intr f\left|\score \frac{f}{g} \right|_{A\ast g}^2\\
    &+ C_{coe}\intr f \jb{v}^{-\gamma}\|A\ast (f-g)\|^2 |\score f|^2 + C_{coe} \intr f \jb{v}^{-\gamma} |b\ast(f-g)|^2. \label{eqn: intermediate inequality on KL}
\end{align}
Since 
\[
    \|A(z)\|  \lesssim \frac{1}{|z|},
\]
the term  $\|A\ast (f-g)\|$ can be controlled by Cauchy-Schwarz as follows
\begin{align}
    \|A\ast (f-g)\|
    &\lesssim \intrthree \frac{1}{|v - w|} |f(w) - g(w)| \d w \\
    &\leq \int_{|v-w|\leq 1} \frac{1}{|v - w|} |f(w) - g(w)| \d w + \lpnorm{f-g}{1}\\
    &\leq \left(\int_{|w|\leq 1} \frac{1}{|w|^2}\d w\right)^{1/2} \lpnorm{f-g}{2} + \lpnorm{f-g}{1}\\
    &\lesssim \lpnorm{f-g}{2} + \lpnorm{f-g}{1}\\
    &\leq \sqrt{2(\lpnorm{f}{\infty} + \lpnorm{g}{\infty})\KL{f}{g}} +\sqrt{2 \KL{f}{g}} ,
\end{align}
where we applied both the classical and $L^2$ version Pinsker inequality as in lemma \ref{lemma: L2 Pinsker}. Thus, the first term in \eqref{eqn: intermediate inequality on KL} becomes
\begin{align}\label{eqn: control of the diffusion term}
    \intrthree f \jb{v}^{3}\|A\ast (f-g)\|^2 |\score f|^2 \lesssim (\lpnorm{f}{\infty}+\lpnorm{g}{\infty}) \lpnorm{f \jb{v}^3 |\score f|^2}{1} \KL{f}{g}.
\end{align}

Now we turn to the second term in \eqref{eqn: intermediate inequality on KL}. Since 
\begin{align}
    |b(z)| &\lesssim \frac{1}{|z|^2},
\end{align}
we obtain that 
\begin{align}
    |b\ast(f-g)| 
    &\lesssim \intrthree \frac{1}{|v-w|^2} |f(w)-g(w)|\d w\\
    &\leq \int_{|v-w|\leq 1} \frac{1}{|v-w|^2} |f(w)-g(w)|\d w + \lpnorm{f-g}{1}\\
    &\leq \intrthree h(v-w) |f(w)-g(w)|\d w + \sqrt{2 \KL{f}{g}}, 
\end{align}
where $h(z) := \frac{1}{|z|^2}1_{|z|\leq 1}$. Now, by Cauchy-Schwarz, for any $0<a<3/4$ (to ensure that $h^{2 a}$ is integrable),
\begin{align}
    \intrthree h(v-w) |f(w)-g(w)|\d w
    &\leq \left(\lpnorm{h^{2a}}{1}\right)^{1/2}   \left(\int_{|v-w|\leq 1} \frac{1}{|v-w|^{2-2a}} |f(w)-g(w)|^2\d w \right)^{1/2}.
\end{align}
Putting the two above estimates together, and applying Young's convolutional inequality with $\frac{1}{p}+\frac{1}{q}=1$,
\begin{align}
    \intrthree f(v) \jb{v}^3 |b \ast (f-g)|^2
    &\lesssim \intrthree f(v) \jb{v}^3 \left(\intrthree h(v-w) |f(w)-g(w)| \d w \d v \right)^2 + \KL{f}{g} \intrthree f(v) \jb{v}^3\\
    &\leq \lpnorm{h^{2a}}{1} \intrthree \intrthree f(v) \jb{v}^3 (h(v-w))^{2-2a} |f(w)-g(w)|^2\d w \d v + \KL{f}{g} \lpnorm{f \jb{v}^3}{1}\\
    &\leq \lpnorm{h^{2a}}{1} \lpnorm{f \jb{v}^3}{p} \lpnorm{h^{2-2a}}{q} \lpnorm{f-g}{2}^2 + \KL{f}{g} \lpnorm{f \jb{v}^3}{1}\\
    &\lesssim \KL{f}{g}  \Big( \lpnorm{h^{2a}}{1} \lpnorm{f \jb{v}^3}{p} \lpnorm{h^{2-2a}}{q} (\lpnorm{f}{\infty} + \lpnorm{g}{\infty}) + \lpnorm{f \jb{v}^3}{1} \Big).
\end{align}
Concretely, we may take $a=\frac{1}{2}, q=2,p=2$, so that $\lpnorm{h^{2a}}{1}$ and $\lpnorm{h^{2-2a}}{q}$ are finite. We obtain
\begin{align}\label{eqn: control of drift term}
    \intrthree f(v) \jb{v}^3 |b \ast (f-g)|^2 \lesssim \KL{f}{g}  \Big( \lpnorm{f \jb{v}^3}{2} (\lpnorm{f}{\infty} + \lpnorm{g}{\infty}) + \lpnorm{f \jb{v}^3}{1} \Big)
\end{align}
Using estimates \eqref{eqn: control of the diffusion term} and \eqref{eqn: control of drift term} in inequality \eqref{eqn: intermediate inequality on KL}, we obtain
\begin{align}\label{eqn: second intermediate inequality on KL}
    \frac{\d}{\d t} \KL{f}{g} \lesssim& -\frac{1}{2} \intr f\left|\score \frac{f}{g} \right|_{A\ast g}^2 \\
    &+ \KL{f}{g} \Big( \left(\lpnorm{f}{\infty}+\lpnorm{g}{\infty}\right) \left(\lpnorm{f \jb{v}^3 |\score f|^2}{1} + \lpnorm{f \jb{v}^3}{2}\right) + \lpnorm{f \jb{v}^3}{1} \Big).
\end{align}
\end{proof}

\section{Main results}\label{sec: main}
The following theorem demonstrates the KL-stability of the spatially homogeneous Landau equation. The assumption of weighted $L^2$ integrability of $g_0$ is only used to guarantee the finiteness of $\sup_{t\geq 0} \lpnorm{g_t}{\infty}$. Otherwise, $g$ can be any positive smooth solution with finite energy and entropy.

\begin{theorem} \label{thm: Landau stability}
Assume that $f$ and $g$ solve the Landau equation with the initial data $f_0$ and $g_0$ that satisfy 
    \begin{align}
      &  \| f_0 \|_{H^2_5}^2 := \intrthree|\nabla^2 f_0(v)|^2 \jb{v}^{10} \d v+ \intrthree |\nabla f_0(v)|^2 \jb{v}^{10}\d v + \intrthree|f_0(v)|^2 \jb{v}^{10} \d v < \infty, \\
      &   \lpnorm{g_0 \jb{v}^{\frac{9}{4}}}{2} < \infty
    \end{align}
    and have finite energy and entropy.
    \begin{align}
        \frac{\d}{\d t}\KL{f_t}{g_t} \leq C (1+t) \KL{f}{g}.
    \end{align}
    for some $C = C(f_0, g_0)$. In particular,
    \begin{align}
        \KL{f_t}{g_t} \leq C e^{t^2} \KL{f_0}{g_0}.
    \end{align}
\end{theorem}

\paragraph{Proof idea.}
We differentiate the relative entropy $\KL{f_t}{g_t}$ along the two Landau flows written in continuity form, obtaining
\[
\frac{\d}{\d t}\KL{f}{g}
    = \intr f \,(U[f]-U[g])\cdot\nabla\log\frac{f}{g}.
\]
Expanding $U[f]-U[g]$ and invoking the coercivity of $A\ast g$ yields a coercive quadratic term 
\[
-\!\intr f\,\big|\nabla\log(f/g)\big|^2_{A\ast g},
\]
which governs the dissipation. The remaining cross terms -- arising from $A\ast(f-g)$ and $b\ast(f-g)$ -- are controlled via Cauchy–Schwarz, the Pinsker inequality, and weighted moment and Fisher-information bounds. Finally, applying Grönwall’s inequality provides an exponential-in-$t^2$ control of the relative entropy, establishing stability.

\begin{proof}
    We will write $f$ and $\intr f$, instead of $f_t(v)$ and $\intr f \d v$ respectively, for simplicity.  The following calculation holds for any values of $d, \gamma$.
    \begin{align}
    \frac{\d}{\d t}\KL{f}{g} 
    &= -\intr \frac{f}{g}\partial_t g  + \intr \log\frac{f}{g}\partial_t f\\ 
    &= -\intr U[g]\cdot \nabla\left(\frac{f}{g}\right) g  + \intr U[f] \cdot \score \left( \frac{f}{g} \right)f \\ 
    &= \intr f (U[f] - U[g]) \cdot \score \frac{f}{g}.
\end{align}
Now, dropping the strictly negative term in the conclusion of lemma \ref{lemma: key calculation}, we obtain
\begin{align}
    \frac{\d}{\d t}\KL{f}{g} 
    &= \intr f (U[f] - U[g]) \cdot \score \frac{f}{g} \\
    &\leq C \KL{f}{g} \Big( \left(\lpnorm{f}{\infty}+\lpnorm{g}{\infty}\right) \left(\lpnorm{f \jb{v}^3 |\score f|^2}{1} + \lpnorm{f \jb{v}^3}{2}\right) + \lpnorm{f \jb{v}^3}{1} \Big)
\end{align}
Finally, by lemma \ref{lemma: sobolev norm of initial data implies finiteness for all t}, using the assumption $\| f_t \|_{H^2_5}^2 < \infty$ on the initial data, we obtain
\begin{align}
    \frac{\d}{\d t}\KL{f}{g} &\leq C (1+t) \KL{f}{g} \sup_{0 \leq s \leq t}\lpnorm{g_s}{\infty}
\end{align}
for some $C = C(f_0)$. The integrability assumption on $g_0$ is sufficient to guarantee $\sup_{0 \leq s \leq t}\lpnorm{g_s}{\infty} < \infty$. Thus, one has
\begin{align}
    \frac{\d}{\d t}\KL{f}{g} &\leq C (1+t) \KL{f}{g}
\end{align}
where $C=C(f_0,g_0)$. Gronwall's inequality completes the proof.
\end{proof}

The following theorem bounds the relative entropy error between the true solution to the Landau equation and the score-based numerical method proposed in \cite{ilin2024transport} and \cite{huang2025score}. The bound is not specific to SBTM and holds for other methods that approximate the score, such as the blob method \cite{carrillo2020landau}.
This theorem is similar to theorem 4.2 in \cite{ilin2024transport} but our theorem considers Coulomb collisions, and bounds $\KL{f}{g}$ instead of $\KL{g}{f}$. The score matching loss $L$ was used to train a neural network to numerically solve the Landau equation. Since the core of the method is to minimize this loss, the bound below provides a quantitative and actionable estimate on the accuracy of the numerical method.
\begin{theorem}\label{thm: numerical method error bound}
    Assume that $f$ solves the Landau equation with initial data satisfying
    \begin{align}
        \| f_0 \|_{H^2_5}^2 < \infty.
    \end{align}
    and $g$ follows the vector field 
    \begin{align}
        U[g_t, s_t](v) := -\intr A(v-w)(s_t(v) - s_t(w)) g_t(w)\d w
    \end{align}
    for some time-dependent vector field $s_t: \R^d \to \R^d$.
    Further, assume that 
    \begin{align}
        \lpnorm{\frac{f_t}{g_t}}{\infty} < \infty, \quad \text{and} \quad \sup_{t\geq 0} \lpnorm{g_t}{\infty} < \infty.
    \end{align}
    Then the KL divergence between $f$ and $g$ can be controlled by the weighted score matching loss:
    \begin{align}
        \frac{\d}{\d t}\KL{f_t}{g_t} &\leq 2\lpnorm{\frac{f_t}{g_t}}{\infty} L(s_t, g_t, A\ast g_t) + C (1+t) \KL{f_t}{g_t},\\
        L(s, g, A\ast g) &:= \intr |\score g(w) - s(w)|^2_{A\ast g} g(w) \d w
    \end{align}
    for some $C = C(f_0, \sup_{0 \leq s \leq t}\lpnorm{g}{\infty})$. In particular, assuming identical initial data $f_0=g_0$, one has
    \begin{align}
        \KL{f_T}{g_T} \leq C e^{T^2} \int_0^T \lpnorm{\frac{f_t}{g_t}}{\infty} L(s_t, g_t, A\ast g_t) \d t.
    \end{align}
\end{theorem}

\begin{proof}
We omit the time subscript for brevity. Similarly to the computation in theorem \ref{thm: Landau stability}, we obtain
    \begin{align}
        \frac{\d}{\d t}\KL{f}{g}
        &= \intr (U[f] - U[g,s]) \cdot \score \frac{f}{g} f \d v \\
        &= \intr (U[g] - U[g,s]) \cdot \score \frac{f}{g} f \d v + \intr (U[f] - U[g]) \cdot \score \frac{f}{g} f \d v.
    \end{align}
The first term is estimated as follows.
    \begin{align}
        \intr (U[g] - U[g,s]) \cdot \score \frac{f}{g} f \d v
        =& \intr \left[ (\score g(v) - s(v)) - (\score g(w) - s(w)) \right]^T A(v-w) \score \frac{f(v)}{g(v)} f(v) g(w) \d w \d v \\
        =& \intr \left[\score g(v) - s(v) \right]^T A(v-w) \score \frac{f(v)}{g(v)} f(v) g(w) \d w \d v \\
        &- \intr \left[\score g(w) - s(w) \right]^T A(v-w) \score \frac{f(v)}{g(v)} f(v) g(w) \d w \d v \\
        \leq& \intr |\score g - s|^2_{A\ast g} f\d v + \frac{1}{4}\intr \left|\score \frac{f}{g}\right|^2_{A\ast g}f \d v\\
        &+ \intr |\score g - s|^2_{A\ast g} f\d v + \frac{1}{4}\intr \left|\score \frac{f}{g}\right|^2_{A\ast g}f \d v\\
        \leq& 2\lpnorm{\frac{f}{g}}{\infty}\intr |\score g - s|^2_{A\ast g} g\d v + \frac{1}{2}\intr \left|\score \frac{f}{g}\right|^2_{A\ast g}f \d v. 
    \end{align}
Now, by lemma \ref{lemma: key calculation}, the second term satisfies
    \begin{align}
        \intr f (U[f] - U[g]) \cdot \score \frac{f}{g} \leq& -\frac{1}{2} \intr f\left|\score \frac{f}{g} \right|_{A\ast g}^2 \\
        &+ C \KL{f}{g} \Big( \left(\lpnorm{f}{\infty}+\lpnorm{g}{\infty}\right) \left(\lpnorm{f \jb{v}^3 |\score f|^2}{1} + \lpnorm{f \jb{v}^3}{2}\right) + \lpnorm{f \jb{v}^3}{1} \Big)
    \end{align}
Finiteness of energy and entropy of $g$ follows from the corresponding finiteness at $t=0$ and energy and entropy conservations shown in proposition 3.1 in \cite{ilin2024transport}.
Now one obtains
\begin{align}
    \frac{\d}{\d t}\KL{f}{g} \leq& \ 2\lpnorm{\frac{f}{g}}{\infty}\intr |\score g - s|^2_{A\ast g} g\d v \\
    &+ C \KL{f}{g} \Big( \left(\lpnorm{f}{\infty}+\lpnorm{g}{\infty}\right) \left(\lpnorm{f \jb{v}^3 |\score f|^2}{1} + \lpnorm{f \jb{v}^3}{2}\right) + \lpnorm{f \jb{v}^3}{1} \Big)
\end{align}
Finally, by lemma \ref{lemma: sobolev norm of initial data implies finiteness for all t}, using the assumption $\| f_t \|_{H^2_5}^2 < \infty$ on the initial data, one has
\begin{align}
    \frac{\d}{\d t}\KL{f}{g} &\leq 2\lpnorm{\frac{f}{g}}{\infty} \intr |\score g - s|^2_{A\ast g} g \d v + C (1+t) \KL{f}{g}
\end{align}
for some $C = C(f_0, \sup_{0 \leq s \leq t}\lpnorm{g}{\infty})$. Gronwall's inequality completes the proof.
\end{proof}

\section*{Acknowledgements}
We thank Zhenfu Wang, Luis Silvestre, Sehyun Ji, Jingwei Hu, Weiran Sun, and Qin Li for helpful discussions about this work.

\printbibliography

\end{document}